\documentclass[10pt,a4paper,english]{article}

\usepackage{hyperref}
\usepackage{setspace}
\usepackage{lipsum}

\usepackage{authblk}
\usepackage[style=alphabetic,sorting=nyt,backend=bibtex,firstinits=true]{biblatex}
\AtBeginBibliography{\setstretch{1}\footnotesize}
\AtEveryBibitem{%
  \iffieldundef{url}{}{\clearfield{doi}}%
}

\usepackage{tikz}
\usetikzlibrary{arrows}
\usetikzlibrary{calc}
\usetikzlibrary{fit}

\usepackage{dsfont}
\usepackage{amsfonts,amssymb,stmaryrd,amsmath,amsthm,amssymb}
\usepackage{graphicx,color}
\usepackage{multimedia}
\usepackage{float}
\usepackage{placeins}

\usepackage{mathrsfs}

\usepackage{algorithm}

\usepackage{mathtools}

\usepackage[a4paper]{geometry}
\newtheorem{theorem}{Theorem}[section]

\newtheorem{proposition}[theorem]{Proposition}

\newtheorem{definition}[theorem]{Definition}

\newcommand \p {\partial}

\newcommand \R {\mathbb{R}}

\renewcommand \L {\mathrm{L}}

\newcommand \LL {\mathbf{L}}

\newcommand \HH {\mathbf{H}}

\newcommand \I {\mathrm{I}}

\newcommand \Id {\mathrm{Id}}

\renewcommand \d {\mathrm{d}}

\DeclareMathOperator{\divg}{div}

\begingroup\makeatletter\ifx\SetFigFont\undefined%
\gdef\SetFigFont#1#2#3#4#5{%
  \reset@font\fontsize{#1}{#2pt}%
  \fontfamily{#3}\fontseries{#4}\fontshape{#5}%
  \selectfont}%
\fi\endgroup%

\title{Approximate controllability of a 2D linear system related to the motion of two fluids with surface tension}

\author[1, 2]{S\'ebastien Court}
\affil[1]{\begin{small}Department of Mathematics, University of Innsbruck, Technikerstrasse 13, 6020 Innsbruck, Austria.\end{small}}
\affil[2]{\begin{small}Digital Science Center, University of Innsbruck, Innrain 15, 6020 Innsbruck, Austria. Email: {\tt sebastien.court@uibk.ac.at}\end{small}}

\begin{document}

\maketitle

\begin{abstract}
We consider a coupled system of partial differential equations describing the interactions between a closed free interface and two viscous incompressible fluids. The fluids are assumed to satisfy the incompressible Navier-Stokes equations in time-dependent domains that are determined by the free interface. The mean curvature of the interface induces a surface tension force that creates a jump of the Cauchy stress tensor on both sides. It influences the behavior of the surrounding fluids, and therefore the deformation of this interface via the equality of velocities. In dimension 2, the steady states correspond to immobile interfaces that are circles with all the same volume. Considering small displacements of steady states, we are lead to consider a linearized version of this system. We prove that the latter is approximately controllable to a given steady state for any time $T>0$ by the means of additional surface tension type forces, provided that the radius of the circle of reference does not coincide with a scaled zero of the Bessel function of first kind.
\end{abstract}

\hfill  


\section{Introduction}
Topics related to the study of the motion and shape of soap bubbles is getting more and more attention, as they have applications in the design of medical treatments and food products for example. Their motion is realized via surface tension, by interaction with surrounding fluids. The question of their stability is a major issue, and the Kolmogorov theory predicts that a soap bubble does not burst as long as the ratio between the inertia forces and the surface tension forces does not reach a power of a critical diameter, quantified by Hinze~\cite{Hinze1955}. At intermediate Reynolds number, this question is delicate as the inertia forces play a non-negligible role, leading for example to bouncing effects~\cite[Fig.~11]{Court_JCAM2019}. In order to circumvent this limitation on the diameter, one can also try to design additional surface tension forces that control and guarantee the stability of the soap bubble. In practice the realization of this type of controls can be realized in different ways, like for example surfactants~\cite{Ferri2000, Rana2012}, or electrochemical effects~\cite{Giant2014, Eaker2016}. In the present paper, we propose sufficient conditions on the diameter of a 2-dimensional soap bubble for being approximately controllable to a circle, without bound constraint on this diameter.\\

Given a fixed bounded domain~$\Omega \subset\R^2$, a regular closed curve describing an interface~$\Gamma(t)$ at time~$t$ can be described with a smooth deformation $X(\cdot,t)$ from a reference configuration~$\Gamma$, such that $\Gamma(t) = X(\Gamma,t)$. It splits~$\Omega$ into two connected parts~$\Omega^+(t)$ and~$\Omega^-(t)$, as described in Figure~\ref{fig-desc}. The reference interface~$\Gamma$ is chosen to be a circle of radius~$r_s>0$ such that $\Gamma \cap \p \Omega = \emptyset$.

\vspace*{-1.5cm}
\begin{minipage}{\linewidth}
\vspace*{1cm}
\begin{center}
\scalebox{0.55}{
\begin{picture}(0,0)
\includegraphics{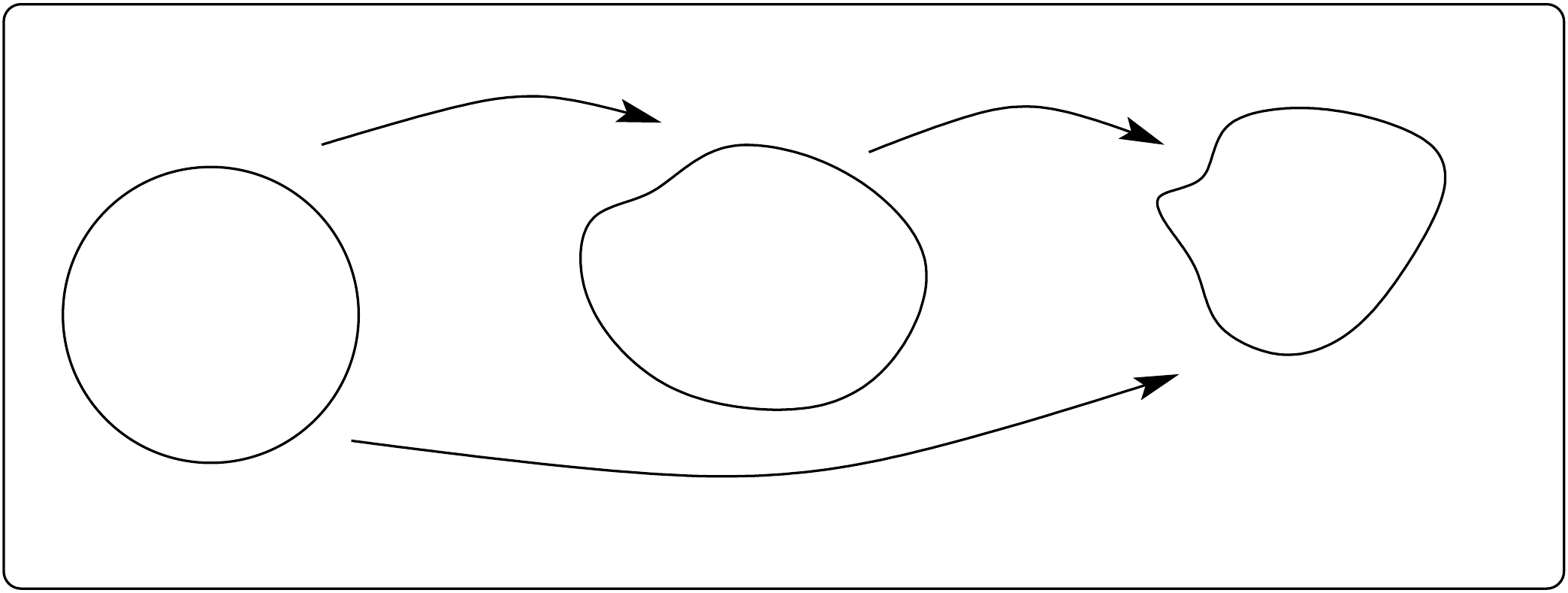}
\end{picture}
\setlength{\unitlength}{4144sp}
\begin{picture}(10887,4564)(1775,-5187)
\put(2800,-3400){\makebox(0,0)[lb]{\smash{{\SetFigFont{20}{16.4}{\rmdefault}{\mddefault}{\updefault}{\color[rgb]{0,0,0}$\Omega^-$}%
}}}}
\put(2000,-2700){\makebox(0,0)[lb]{\smash{{\SetFigFont{20}{16.4}{\rmdefault}{\mddefault}{\updefault}{\color[rgb]{0,0,0}$\Omega^+$}%
}}}}
\put(7336,-3616){\makebox(0,0)[lb]{\smash{{\SetFigFont{20}{16.4}{\rmdefault}{\mddefault}{\updefault}{\color[rgb]{0,0,0}$X_0(\Gamma)$}%
}}}}
\put(6000,-5000){\makebox(0,0)[lb]{\smash{{\SetFigFont{20}{16.4}{\rmdefault}{\mddefault}{\updefault}{\color[rgb]{0,0,0}$X(\cdot,t)$}%
}}}}
\put(4300,-2000){\makebox(0,0)[lb]{\smash{{\SetFigFont{20}{16.4}{\rmdefault}{\mddefault}{\updefault}{\color[rgb]{0,0,0}$X_0$}%
}}}}
\put(7000,-2000){\makebox(0,0)[lb]{\smash{{\SetFigFont{20}{16.4}{\rmdefault}{\mddefault}{\updefault}{\color[rgb]{0,0,0}$X(\cdot,t) \circ X_0^{-1}$}%
}}}}
\put(4000,-3700){\makebox(0,0)[lb]{\smash{{\SetFigFont{20}{16.4}{\rmdefault}{\mddefault}{\updefault}{\color[rgb]{0,0,0}$\Gamma$}%
}}}}
\put(10200,-3500){\makebox(0,0)[lb]{\smash{{\SetFigFont{20}{16.4}{\rmdefault}{\mddefault}{\updefault}{\color[rgb]{0,0,0}$\Gamma(t)$}%
}}}}
\put(9200,-3000){\makebox(0,0)[lb]{\smash{{\SetFigFont{20}{16.4}{\rmdefault}{\mddefault}{\updefault}{\color[rgb]{0,0,0}$\Omega^-(t)$}%
}}}}
\put(10200,-2200){\makebox(0,0)[lb]{\smash{{\SetFigFont{20}{16.4}{\rmdefault}{\mddefault}{\updefault}{\color[rgb]{0,0,0}$\Omega^+(t)$}%
}}}}
\put(11500,-1600){\makebox(0,0)[lb]{\smash{{\SetFigFont{24}{16.4}{\rmdefault}{\mddefault}{\updefault}{\color[rgb]{0,0,0}$\Omega$}%
}}}}
\end{picture}
}
\end{center}
\vspace*{-0.5cm}
\begin{figure}[H]
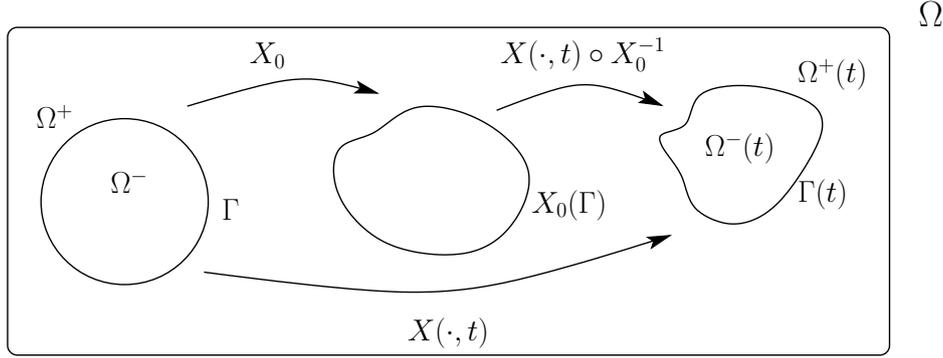

\caption{Description of the geometric configuration, and related notation. \label{fig-desc}}
\end{figure}
\end{minipage}
\FloatBarrier
\hfill \\

We will use the notation of type~$\Omega^\pm(t)$ when refering to~$\Omega^-(t)$ and~$\Omega^-(t)$ simultaneously and respectively. The initial system we are interested in couples the incompressible Navier-Stokes equations with the free closed interface~$\Gamma(t)$, as follows:

\begin{subequations} \label{sys-origin}
\begin{eqnarray}
\rho \left(\frac{\p u}{\p t} + (u\cdot\nabla)u\right)
- \divg \sigma(u,p) = \rho f \quad \text{ and } \quad \divg u = 0 & & \text{in } \Omega^\pm(t), \ t \in (0,T), \label{eq-main-origin1}\\
u^+ = 0 & &  \text{on } \p \Omega \times (0,T), \\
u^\pm = \frac{\p X}{\p t}\left( X(\cdot,t)^{-1},t\right) 
\quad \text{ and } \quad
-\left[ \sigma(u,p)\right]n = \mu \kappa n + g & & \text{on } \Gamma(t), \ t \in (0,T), \label{eq-main-origin3}\\
X(\cdot,0) = X_0 \text{ on } \Gamma,\quad
u(\cdot,0) = u_0 \text{ in } \Omega^\pm(0). & & 
\end{eqnarray}
\end{subequations}
The interface~$\Gamma(t)$, defined by mapping~$X(\cdot,t)$, splits the domain~$\Omega$ into two fluid domains~$\Omega^+(t)$ and~$\Omega^-(t)$, in which the Navier-Stokes equations are considered, with velocity/pressure couples~$(u,p)$ as unknowns. The equality of velocities on~$\Gamma(t)$ (first equation of~\eqref{eq-main-origin3}) corresponds to the Eulerian-Lagrangian correspondence. The outword normal vector of~$\Gamma(t)$ is denoted by~$n$, and the mean curvature~$\kappa$ of~$\Gamma(t)$ induces a jump of the normal trace of the Cauchy stress tensor~$\sigma(u,p) = 2\nu\varepsilon(u) - p\I = \nu(\nabla u +\nabla u^T) -p \I$ (second equation of~\eqref{eq-main-origin3}), where~$\mu >0$ is a constant surface tension coefficient. This jump expresses as $\left[\sigma(u,p) \right] = (2\nu^+\varepsilon(u^+) -p^+\I) - (2\nu^-\varepsilon(u^-) -p^-\I)$. The viscosity~$\nu$ is assumed to be constant inside~$\Omega^\pm(t)$, equal to $\nu^\pm$, as well as the density~$\rho$ equal to~$\rho^\pm$. \\
This type of models has been studied by Denisova and Solonnikov~\cite{Denisova1994, Denisova1995}, based on earlier contributions~\cite{Denisova1989, Denisova1990, Solonnikov1991, Denisova1991, Denisova1993}, after initial investigations of Rivkind~\cite{Rivkind1976, Rivkind1977, Rivkind1979}. Recently the wellposedness questions were revisited in the context of the $L^p$-maximal regularity~\cite{Simonett2009, Simonett2010, Simonett2011}. Further investigations in this direction are given in~\cite{Denisova2021}. From a control perspective, very few mathematical contributions exist, when dealing with jump conditions of type~\eqref{eq-main-origin3}, except~\cite{Court2022} which addresses the low-Reynolds number case (namely~$\rho=0$). \\
Linearized around some steady state $(u_s^\pm,\nabla p_s^\pm,X_s) = (0,0, X_s)$, such that $X_s(\Gamma)$ is also a circle of radius~$r_s$ (because of incompressibility) system~\eqref{sys-origin} becomes the following one:
\begin{subequations} \label{sysmain}
\begin{eqnarray}
\displaystyle \rho^\pm\frac{\p u^\pm}{\p t}
-\divg \sigma^\pm(u^\pm, p^\pm) = \rho^\pm f^\pm \quad \text{ and } \quad 
\divg u^\pm = 0 & & \text{in } \Omega^\pm \times (0,T),  \\
u^+ = 0 & &\text{on } \p \Omega \times (0,T),  \\
u^+ = u^- = \frac{\p Z}{\p t}
\quad \text{ and } \quad 
-\left[\sigma(u,p) \right]n = \mu \Delta_{\Gamma} Z + g & &
\text{on } \Gamma \times (0,T), \\
Z(\cdot,0 ) = Z_0  \text{ on } \Gamma, \quad
u^\pm(\cdot,0) = u_0^\pm \text{ in } \Omega^\pm. & & 
\end{eqnarray}
\end{subequations}
The unknowns are then $(u^\pm,p^\pm,Z)$, where~$Z = X-X_s$ plays the role of a displacement. The forces~$f^\pm$ are given, and $g$ is our control function. At low Reynolds number ($\rho^\pm=0$), an operator formulation was given in~\cite{Court2022} for the corresponding linear system, and feedback laws were derived, based on approximate controllability, which is quite straightforward to prove in this case. Approximate controllability is obtained via a unique continuation argument. In the present paper, we prove approximate controllability of the interface displacement~$Z$ -- up to translations -- when~$\rho^\pm$ are not equal to zero, leading to a less trivial unique continuation result, as in particular it involves zeros of Bessel functions. 
We detail the paper as follows: In section~\ref{sec-steady} we describe the steady states of system~\eqref{sys-origin}. Next section~\eqref{sec-lin} explains how we obtain the linearized system~\eqref{sysmain}. Section~\ref{sec-cont} is dedicated to the approximate controllability of system~\eqref{sysmain}.

\section{On the steady states} \label{sec-steady}

Consider a reference configuration given by a circle~$\Gamma$ of radius $r_s>0$, splitting the domain~$\Omega$ into~$\Omega^+$ and~$\Omega^-$, such that $\Gamma \cap \p \Omega = \emptyset$, we define the following class of displacements:
\begin{equation*}
\mathcal{C} := \left\{
Z_s\in \HH^{1/2}(\Gamma) \text{ such that } (\Id+Z_s)(\Gamma) \text{ is a circle of radius $r_s$ and }
(\Id+Z_s)(\Gamma)\cap \p \Omega = \emptyset
\right\}.
\end{equation*}
Note that the set~$\mathcal{C}$ contains rotations and certain translations, as well as diffeomorphisms of the circle~$\Gamma$. Let us show that the set~$\mathcal{C}$ corresponds to steady deformations~$\Id+Z_s$. The steady states~$(u_s^+,p_s^+, u_s^-,p_s^-, Z_s)$ define $\Gamma_s := (\Id+Z_s)(\Gamma)$ with curvature denoted by~$\kappa_s$,which  splits the domain~$\Omega$ into $\Omega^+_s$ and~$\Omega_s^-$, and satisfy the following system
\begin{subequations} \label{sys-steady}
\begin{eqnarray}
\rho^\pm(u_s^\pm \cdot \nabla ) u_s^\pm - \divg \sigma^\pm(u_s^\pm,p_s^\pm) = 0 
\quad \text{and} \quad
\divg u_s^\pm = 0
& & \text{in } \Omega_s^\pm, \label{sys-steady1}\\
u_s^+ = 0 & & \text{on } \p \Omega, \\
u_s^+ = u_s^- = 0 \quad \text{and} \quad 
-\left[\sigma(u_s,p_s)\right]n_s =  \mu \kappa_sn_s & &
\text{on } \Gamma_s. \label{eq-steady3}
\end{eqnarray}
\end{subequations}
We prove that locally the steady states are trivial states, where the steady displacements~$Z_s$ lie in~$\mathcal{C}$.

\begin{proposition}
There exists a constant $C_s>0$ depending only on $\nu^\pm$, $\rho^\pm$ and $\Omega$, such that the only velocities solution of system~\eqref{sys-steady} in the set $\{ (u_s^+,u_s^-) \in \mathbf{H}^1(\Omega_s^+)\times \HH^1(\Omega_s^-)  | \ \|u^\pm_s\|_{\mathbf{L}^2(\Omega)} \leq C_s \}$ is $u^\pm_s = 0$. This corresponds to a configuration where $\Gamma_s$ is a circle, and the pressures~$p_s^\pm$ are constant on both sides of $\Gamma_s$.
\end{proposition}

\begin{proof}
Recall that $(v \cdot \nabla) v = \divg(v\otimes v)$ when $\divg v = 0$. From~\eqref{sys-steady1}, we deduce by integration by parts
\begin{equation*}
\nu^\pm \| \nabla u_s^\pm \|^2_{[\LL^2(\Omega_s^\pm)]^2}  =  \rho^\pm \int_{\Omega} (u^\pm_s \otimes u^\pm_s): \nabla u^\pm_s \d \Omega_s^\pm.
\end{equation*}
We estimate the right-hand-side with the H\"{o}lder's inequality and the continuous embedding $\HH^{1/2}(\Omega_s^\pm) \hookrightarrow \LL^4(\Omega_s^\pm)$ valid in dimension~2:
\begin{equation*}
\int_{\Omega} (u_s^\pm \otimes u_s^\pm): \nabla u_s^\pm \d \Omega_s^\pm 
 \leq  
\| \nabla u_s^\pm \|_{[\LL^2(\Omega_s^\pm)]^2}\|u_s^\pm \|^2_{\LL^4(\Omega_s^\pm)} 
 \leq  C\| \nabla u_s^\pm \|_{[\LL^2(\Omega_s^\pm)]^2}\|u_s^\pm \|^2_{\HH^{1/2}(\Omega_s^\pm)} .
\end{equation*}
Since we have the interpolation $\HH^{1/2}(\Omega_s^\pm) = \left[ \LL^2(\Omega_s^\pm); \HH^1(\Omega_s^\pm) \right]_{1/2}$, this yields
\begin{equation*}
\nu^\pm  \| \nabla u_s^\pm \|^2_{[\LL^2(\Omega_s^\pm)]^2}  \leq  \rho^\pm C
\| \nabla u_s^\pm \|_{[\LL^2(\Omega_s^\pm)]^2}\|u_s^\pm \|_{\LL^{2}(\Omega_s^\pm)}
 \|u_s^\pm \|_{\HH^{1}(\Omega_s^\pm)}
\leq \rho^\pm C \|u_s^\pm\|_{\LL^2(\Omega_s^\pm)}\|u_s^\pm\|^2_{\HH^1(\Omega_s^\pm)}.
\end{equation*}
Using the Poincar\'e inequality in~$\Omega_s^\pm$, namely $\| u_s^\pm \|_{\HH^1(\Omega_s^\pm)} \leq C_{\mathcal{P}} \|\nabla u_s^\pm\|_{[\LL^2(\Omega_s^\pm)]^2}$, as ${u_s}^\pm_{|\p \Omega_s^\pm} = 0$, we deduce
\begin{equation*}
\nu^\pm \|  u_s^\pm \|^2_{\HH^1(\Omega_s^\pm)} 
\leq \rho^\pm CC_{\mathcal{P}} \|u_s^\pm\|_{\LL^2(\Omega_s^\pm)}\|u_s^\pm\|^2_{\HH^1(\Omega_s^\pm)}
\leq \rho^\pm CC_{\mathcal{P}}C_s \|u_s^\pm\|^2_{\HH^1(\Omega_s^\pm)}.
\end{equation*}
By choosing $C_s < \displaystyle\frac{\nu^\pm}{\rho^\pm C C_{\mathcal{P}}} $, we obtain necessarily~$\|u_s^\pm\|_{\HH^1(\Omega_s^\pm)} = 0$. Hence $u_s^\pm = 0$ almost everywhere in $\Omega_s^\pm$, and so $\nabla p_s^\pm = 0$ in $\Omega_s^\pm$. The pressure is thus constant on both sides of the bubble-soap, and the transmission condition~\eqref{eq-steady3} reduces to the so-called Young-Laplace equation, that is $\kappa_s = [p_s]/\mu$ is constant. So~$\Gamma_s$ is a circle enclosing the same volume as~$\Gamma$ (because of incompressibility), and the displacements~$Z_s$ describe the set~$\mathcal{C}$, completing the proof.
\end{proof}

Therefore system~\eqref{sys-origin} is linearized around $(u_s^+,p_s^+, u_s^-,p_s^-, Z_s)$ such that $u_s^\pm = 0$, $\nabla p_s^\pm = 0$ and $Z_s \in \mathcal{C}$.

\section{On the linearization} \label{sec-lin}

In order to linearize system~\eqref{sys-origin}, the method commonly used consists in defining a change of variable (as an extension of mapping~$X$), introduce a change of unknowns and rewrite the equations~\eqref{eq-main-origin1}-\eqref{eq-main-origin3} in reference domains~$\Omega^\pm$ and~$\Gamma$. Next, introduce the displacement~$Z= X-X_s$, where~$X_s- \Id \in \mathcal{C}$ corresponds to a desired steady state, and linearize the corresponding transformed system by considering~$Z$ small. Referring to~\cite{Court2022}, we then obtain
\begin{equation}
\begin{array} {rcl}
\displaystyle \rho^\pm\frac{\p u^\pm}{\p t}
-\divg \sigma^\pm(u^\pm, p^\pm) = \rho^\pm f^\pm \quad \text{ and } \quad 
\divg u^\pm = 0 & & \text{in } \Omega^\pm \times (0,T),  \\
u^+ = 0 & &\text{on } \p \Omega \times (0,T),  \\
u^+ = u^- = \displaystyle\frac{\p Z}{\p t} 
\quad \text{ and } \quad 
-\left[\sigma(u,p) \right]n = \mu \divg_{\Gamma}( \nabla^{n}_{\Gamma} Z) + g & &
\text{on } \Gamma \times (0,T), \\[5pt]
Z(\cdot,0 ) = Z_0  \text{ on } \Gamma, \quad
u^\pm(\cdot,0) = u_0^\pm \text{ in } \Omega^\pm, & & 
\end{array}
\label{sys-tempo}
\end{equation}
where we have introduced~$\nabla^{n}_{\Gamma} Z := (n\otimes n) \nabla_{\Gamma} Z$.

\subsection{The kernel of~$\nabla^{n}_{\Gamma}$ and the lack of coercivity}
\label{sec-kernel}
The kernel of the operator~$Z\mapsto \nabla^{n}_{\Gamma} Z = (n\otimes n)\nabla_{\Gamma}Z$ which appears in the linearization of $(\kappa n) \circ X$ conditions the unique continuation argument used in the proof of approximate controllability (section~\ref{sec-cont}). Actually there exists smooth mappings~$X = Z+\Id$ that are orientation-preserving, volume-preserving, that transform the circle~$\Gamma$ into another Jordan curve, for which we can choose~$Z$ arbitrarily close to zero -- or any steady state~$Z_s \in \mathcal{C}$ -- and such that~$\divg_{\Gamma}(\nabla^{n}_{\Gamma}Z ) = 0$. For example, considering a smooth parameterization~$X_p: \xi \in [0,2\pi) \rightarrow \Gamma $, the following displacements
\begin{equation*}
Z(X_p(\xi)) =  \sum_{k=2}^{\infty} \frac{1}{k^2-1}
\Big(
\big(
a_{k} \cos(k\xi/r) + b_{k} \sin(k\xi/r)
\big) n +
\big(
-b_{k}k \cos(k\xi/r) + a_{k}k \sin(k\xi/r)
\big)\tau
\Big)
\end{equation*}
(where $\tau$ is the tangent vector of~$\Gamma$) lie in the kernel of~$\nabla^{n}_{\Gamma}$. Several examples of such displacements are represented in Figure~\ref{fig-bubble} below, corresponding to $k \in \{2,3,4,5,6,7\}$ and coefficients $a_k, b_k \in \{0,1\}$.\\
\begin{minipage}{0.96\linewidth}
\centering
\begin{tabular} {c|c|c|c|c|c}
\hspace*{0.0cm}
\begin{minipage}{0.16\linewidth}
\begin{figure}[H]
\includegraphics[trim = 5cm 9cm 5cm 9cm, clip, scale=0.14]{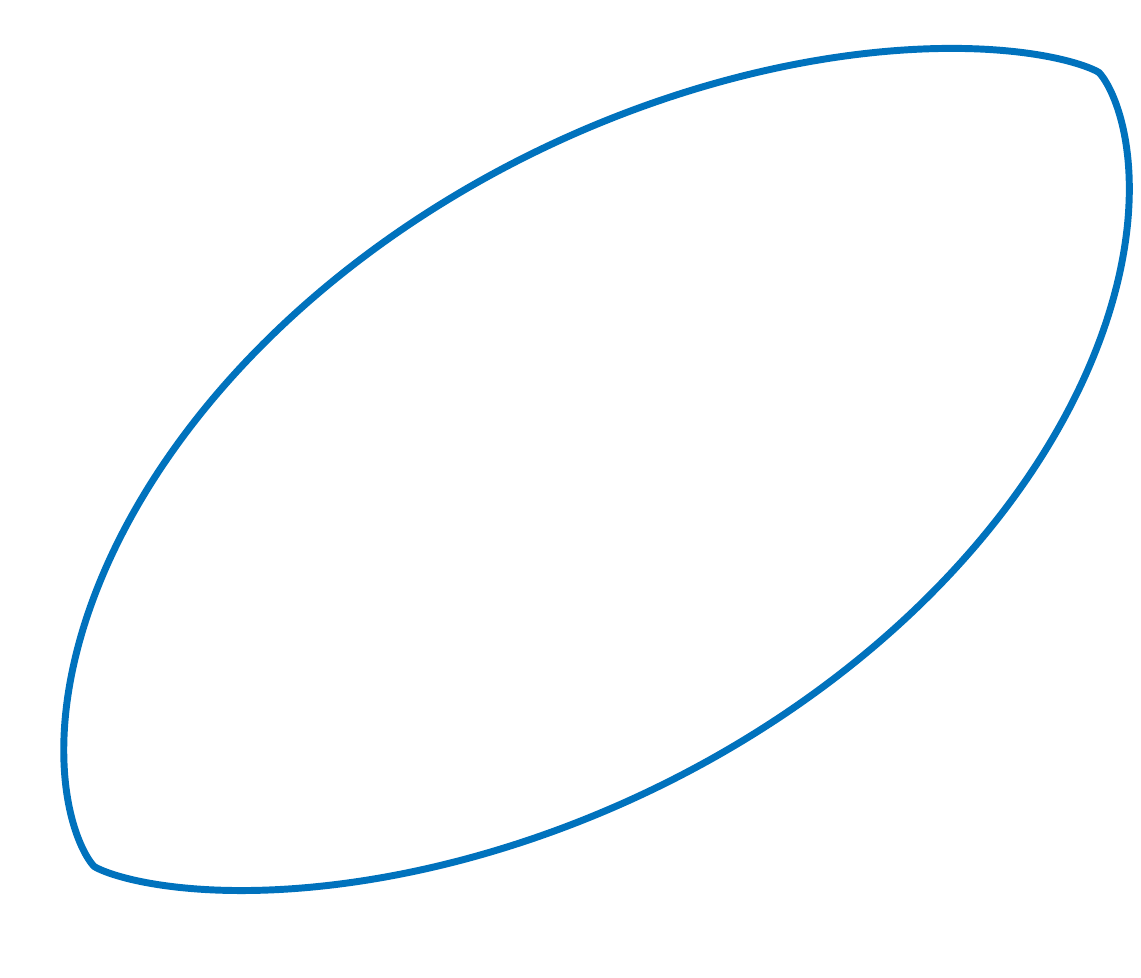}
\end{figure}
\end{minipage}\hspace*{-15pt}
&
\hspace*{5pt}
\begin{minipage}{0.16\linewidth}
\begin{figure}[H]
\includegraphics[trim = 5cm 9cm 5cm 9cm, clip, scale=0.16]{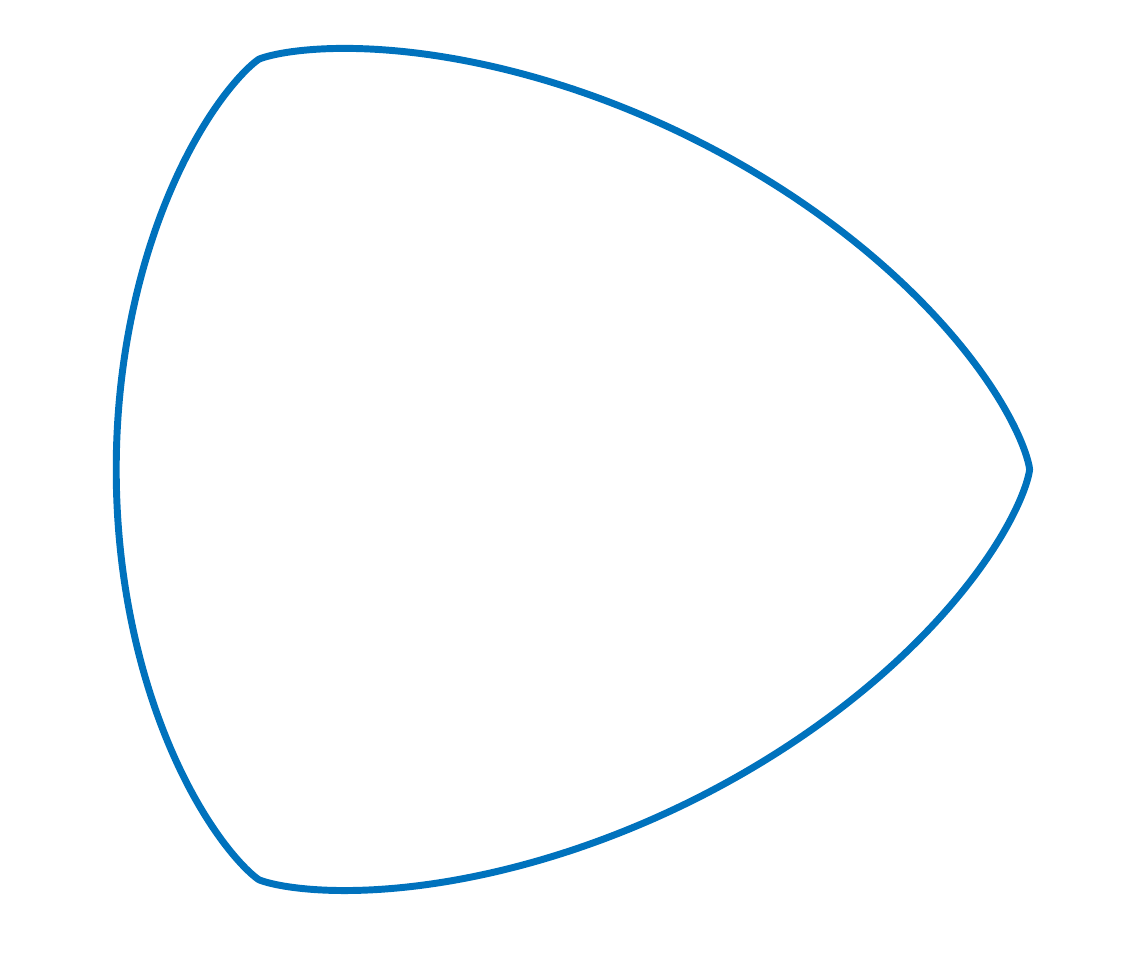}
\end{figure}
\end{minipage}\hspace*{-10pt}
&
\begin{minipage}{0.16\linewidth}
\begin{figure}[H]
\includegraphics[trim = 3cm 9cm 5cm 9cm, clip, scale=0.16]{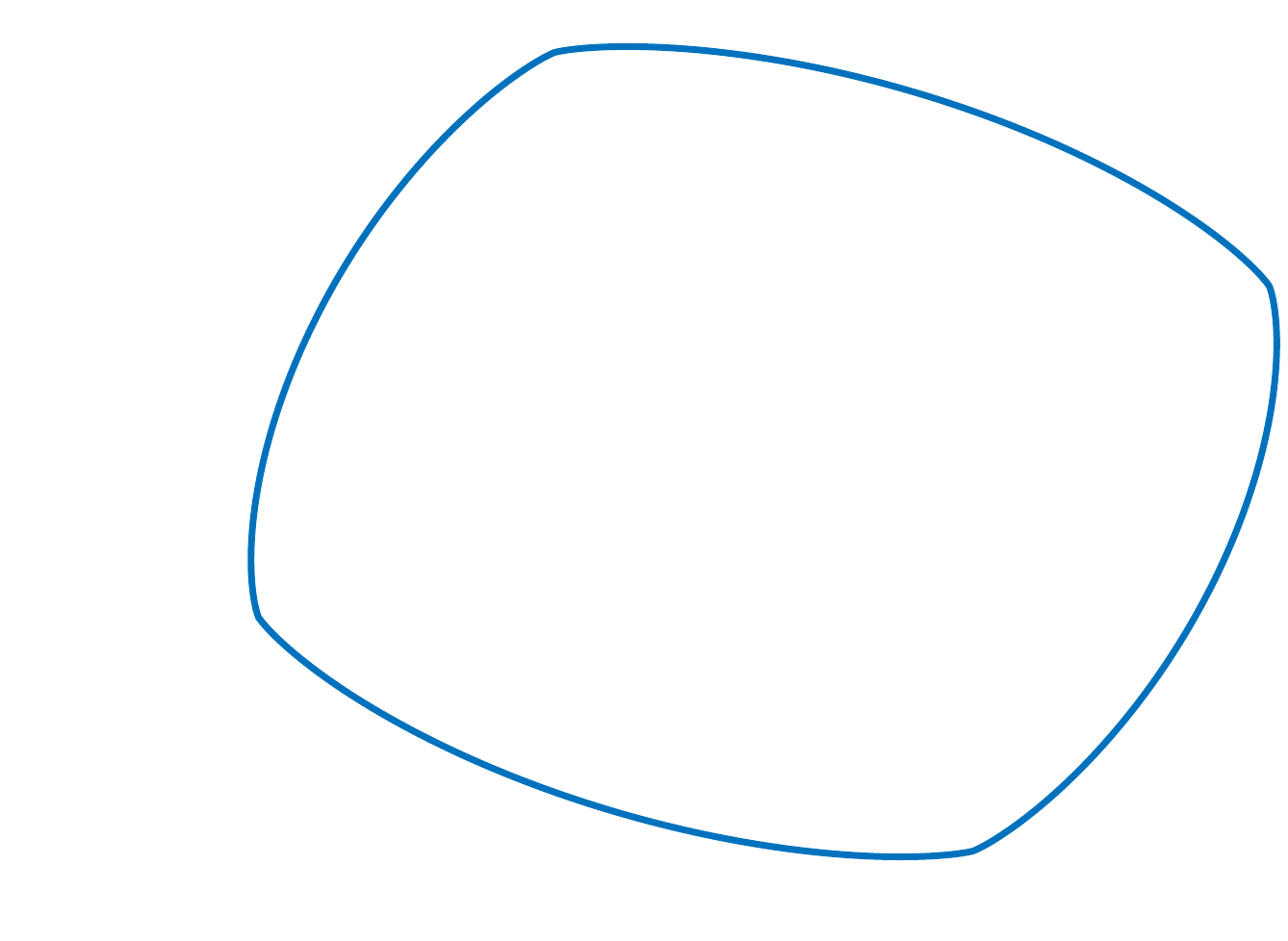}
\end{figure}
\end{minipage}
&
\hspace*{5pt}
\begin{minipage}{0.16\linewidth}
\begin{figure}[H]
\includegraphics[trim = 5cm 9cm 4cm 9cm, clip, scale=0.14]{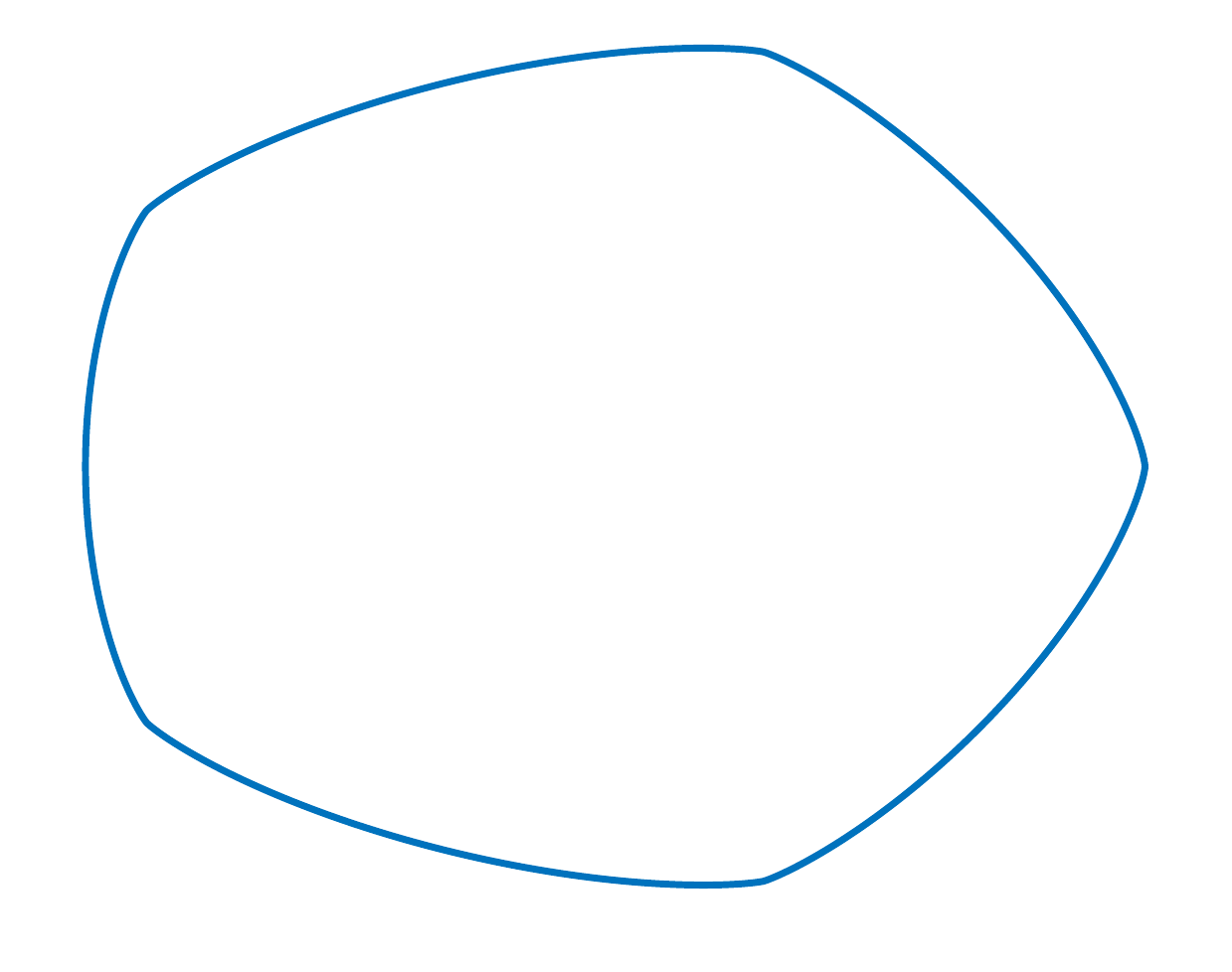}
\end{figure}
\end{minipage}\hspace*{-15pt}
&
\hspace*{5pt}
\begin{minipage}{0.16\linewidth}
\begin{figure}[H]
\includegraphics[trim = 5cm 9cm 4.5cm 9cm, clip, scale=0.14]{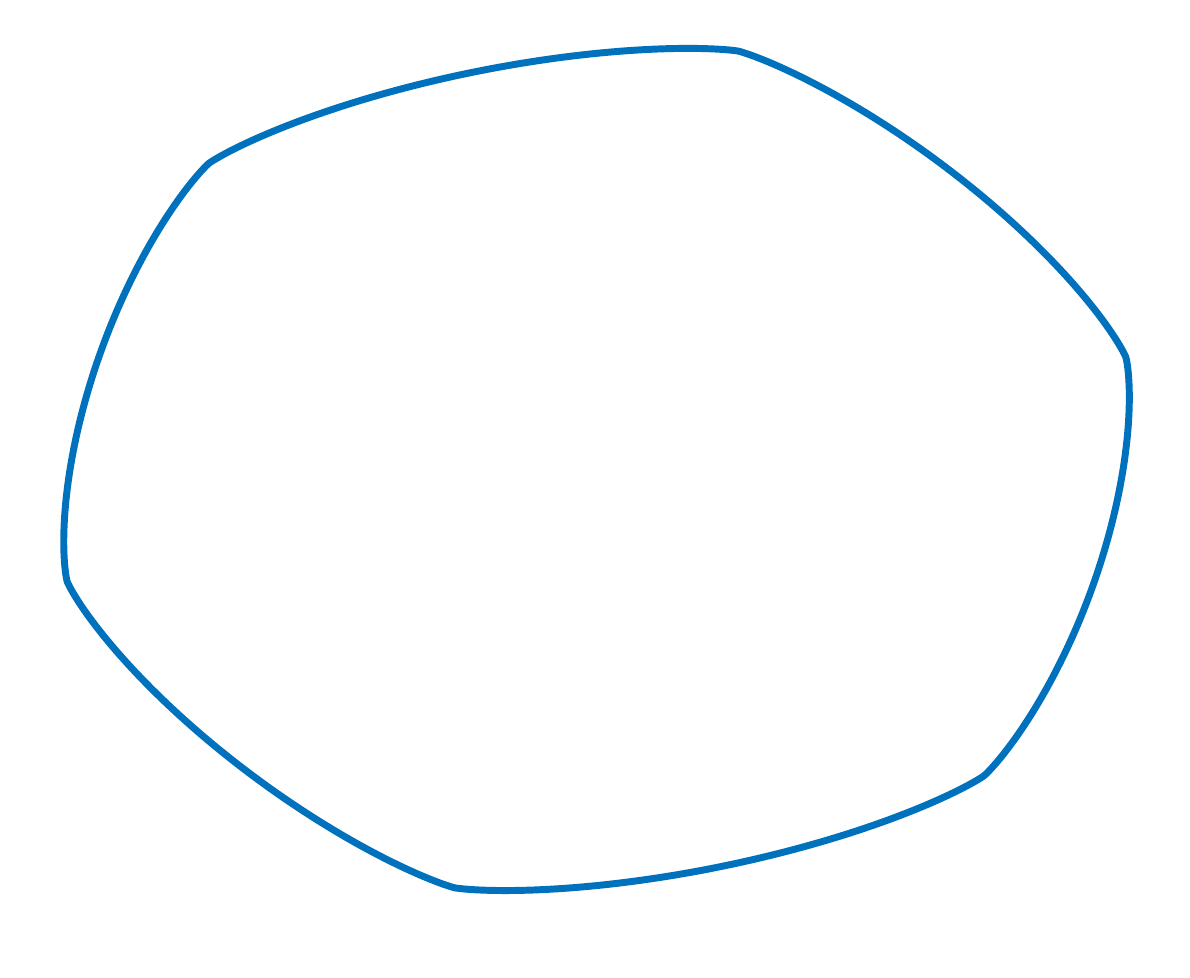}
\end{figure}
\end{minipage}\hspace*{-15pt}
&
\hspace*{5pt}
\begin{minipage}{0.16\linewidth}
\begin{figure}[H]
\includegraphics[trim = 5cm 9cm 5cm 9cm, clip, scale=0.14]{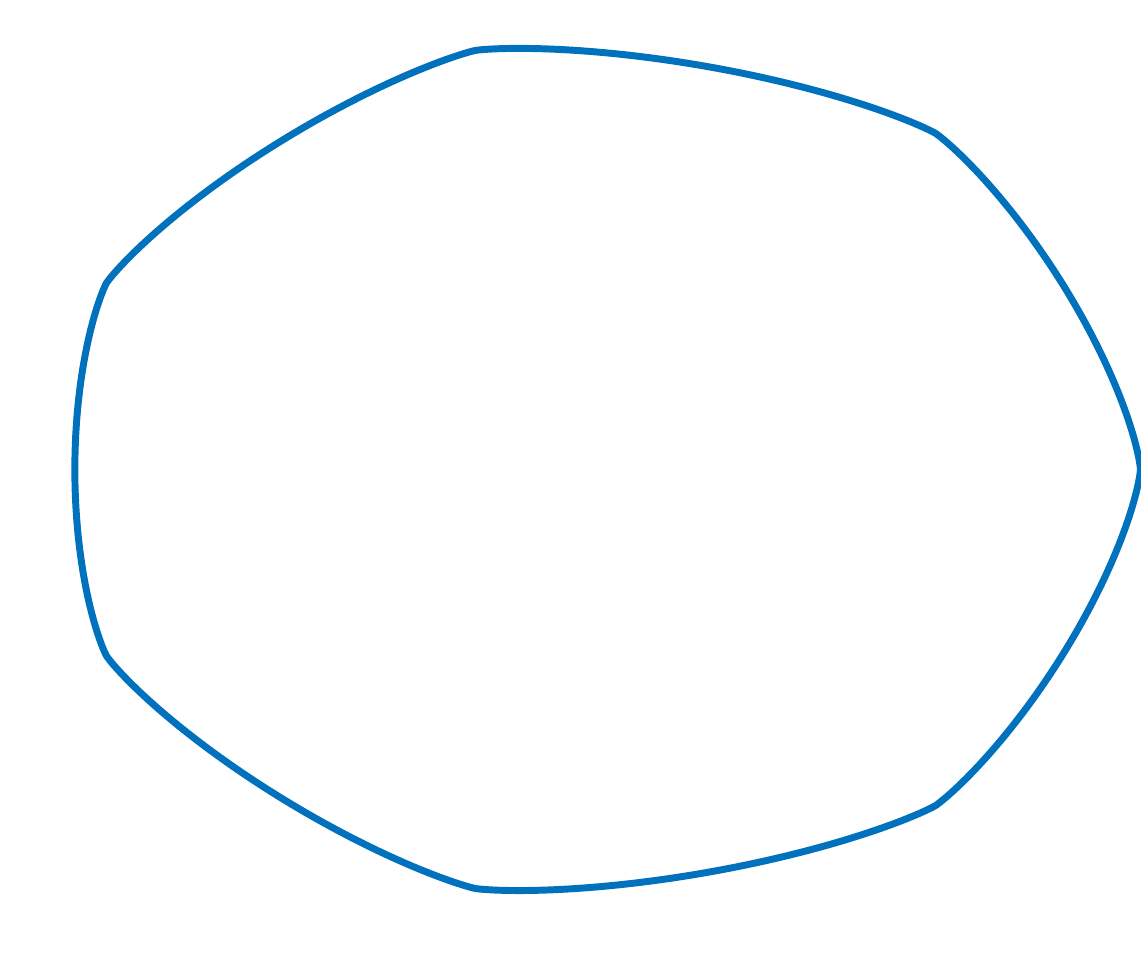}
\end{figure}
\end{minipage}
\end{tabular}
\vspace*{-10pt}
\begin{figure}[H]
\caption{Different deformations $X(\Gamma)$ of the circle such that~$\nabla^{n}_{\Gamma} (X-\Id) =0$.\label{fig-bubble}}
\end{figure}
\end{minipage}\vspace*{5pt}\\ 
Therefore, in the proof of the approximate controllability result in section~\ref{sec-cont}, the orthogonal of the reachable set for system~\eqref{sys-tempo} would not be reduced to trivial displacements, and thus the result would fail.

\subsection{Considering the Laplace-Beltrami operator}
In order to define a linear operator whose the kernel is reduced to trivial displacements, we need to regularize the operator~$\divg_{\Gamma}\nabla^{n}_{\Gamma}$. One possibility consists in considering a prior control~$g_0$, in a form of a feedback operator, as it is done in~\cite{Court2022}, as $g_0= \divg((\tau\otimes \tau)\nabla_{\Gamma} Z)$, dealing with the tangential gradient of~$Z$ only, so that the resulting linear operator corresponds to $\divg_{\Gamma} (\nabla^{n}_{\Gamma} Z) + g_0 = \Delta_{\Gamma} Z$. Note that in~\cite{Denisova1989, Denisova1990, Denisova1993} which deal with linear systems related to system~\eqref{sys-origin}, it is actually the Laplace-Beltrami operator which is considered as operator representing linear version of the surface tension operator. Therefore, from now, we propose to rather deal with system~\eqref{sysmain}, involving the Laplace-Beltrami operator. Alternatively, one could also keep the degenerate operator~$\nabla_{\Gamma}^{n}$, consider a reachable set that excludes the elements of its kernel, and thus approximate controllability would be obtained for system~\eqref{sysmain} up to elements of the kernel of~$\nabla_{\Gamma}^{n}$.\\
For system~\eqref{sysmain}, we will assume in the rest of the paper that for all 
\begin{equation*}
f^\pm \in \mathrm{L}^2(0,T;\LL^2(\Omega^\pm)), \quad 
g \in \mathrm{L}^2(0,T; \HH^{1/2}(\Gamma)) \cap \mathrm{H}^{1/4}(0,T;\LL^2(\Gamma)),
\end{equation*}
there exists a unique solution~$(u^+,u^-,Z)$, satisfying
\begin{equation*}
\displaystyle u^\pm \in \mathrm{L}^2(0,T;\HH^2(\Omega^\pm)) \cap \mathrm{H}^1(0,T;\LL^2(\Omega^\pm)), \quad
\displaystyle \frac{\p Z}{\p t} \in \mathrm{L}^2(0,T;\HH^{3/2}(\Gamma)) \cap \mathrm{H}^1(0,T;\HH^{-1/2}(\Gamma)).
\end{equation*}
This parabolic regularity result is given in~\cite{Simonett2011}. 

\section{Approximate controllability} \label{sec-cont}

Assume now~$f^\pm = 0$, $u_0^\pm = 0$ and~$Z_0=0$. Let us define precisely what we call {\it approximate controllability} in our context.
\begin{definition} \label{def-cont}
Define the reachable set for system~\eqref{sysmain} at time~$T>0$:
\begin{equation*}
\mathcal{R}(T) = \left\{
(u^\pm(T), Z(T)) \in \HH^1(\Omega^\pm) \times \HH^{1/2}(\Gamma)/\R^2 \mid (u^+,u^-,Z) \text{ satisfies~\eqref{sysmain}, } g\in \L^2(0,T;\HH^{1/2}(\Gamma))
\right\}.
\end{equation*}
We say that system~\eqref{sysmain} is approximately controllable at time~$T$ if~$\mathcal{R}(T)$ is dense in~$\HH^1(\Omega^+) \times \HH^1(\Omega^-) \times \HH^{1/2}(\Gamma)/\R^2$.
\end{definition}

Note that the displacements are considered up to translations of~$\R^2$. We prove that system~\eqref{sysmain} is approximately controllable in the sense of Definition~\ref{def-cont}, provided that the radius of the circle of reference does not meet scaled zeros of a Bessel function.

\begin{theorem} \label{th-dim2}
Assume that the radius $r_s>0$ of $\Gamma$ satisfies
\begin{equation}
r_s \notin \left\{ \sqrt{\frac{\nu^-}{\lambda}}r_0 ; \ \lambda \in \mathrm{Sp}(A^-), \ r_0 \in J_1^{-1}(\{0\})\right\},
\label{def-excluded}
\end{equation}
where~$A^-$ denotes the Stokes operator in~$\Omega^-$. Then system~\eqref{sysmain} is approximately controllable in the sense of Definition~\ref{def-cont}, for all $T>0$.
\end{theorem}

\begin{proof}
The method is inspired by~\cite{Court_EECT1}. Consider $(\phi^+_T, \phi^-_T, Z_T) \in \mathcal{R}(T)^\perp$, and introduce the following adjoint system
\begin{equation}
\begin{array}{rcl}
\displaystyle -\rho^\pm\frac{\p \phi^\pm}{\p t} - \divg \sigma^\pm(\phi^\pm, \psi^\pm) = 0 
\quad \text{and} \quad \divg \phi^\pm = 0 & & 
\text{in } \Omega^\pm \times (0,T)\\[5pt]
\phi^+ = 0 & & \text{on } \p \Omega \times (0,T)\\
\phi^\pm =  \displaystyle -\frac{\p \zeta}{\p t} 
\quad\text{and} \quad
-\left[ \sigma (\phi, \psi)\right]n =  \mu \Delta_{\Gamma} \zeta
& & \text{on } \Gamma \times (0,T), \\[5pt]
\phi^\pm(T) = \phi^\pm_T \text{ in }\Omega^\pm, 
\quad  \zeta(T) = Z_T \text{ on }\Gamma,
\end{array} \label{sysadj}
\end{equation}
with $(\phi^+,\psi^+, \phi^-, \psi^-, \zeta)$ as unknowns. Now consider $(u^+, p^+, u^-,p^-, Z)$ solution of~\eqref{sysmain} with $g$ as data. Taking the inner product in $\L^2(0,T;\LL^2(\Omega^\pm))$ of the first equation of~\eqref{sysmain} by $u^\pm$, by integration by parts we obtain for all $g \in \L^2(0,T;\HH^{1/2}(\Gamma))$ the following identity
\begin{equation*}
\rho^+\int_{\Omega^+} \phi^+_T \cdot u^+(T) \, \d \Omega^+ 
+ \rho^-\int_{\Omega^-} \phi^-_T \cdot u^-(T) \, \d \Omega^- 
- \mu\left\langle  \nabla Z_T; \nabla Z(T)\right\rangle_{\LL^2(\Gamma)} = 
-\int_0^T \left\langle g; \frac{\p \zeta}{\p t}\right\rangle_{\LL^2(\Gamma)} \d t.
\end{equation*}
Since $(\phi^+_T, \phi^-_T, Z_T) \in \mathcal{R}(T)^\perp$, this identity yields $\displaystyle\frac{\p \zeta}{\p t} = 0$ on $\Gamma$. Then deriving system~\eqref{sysadj} in time yields
\begin{equation*}
\begin{array}{rcl}
\displaystyle -\rho^\pm\frac{\p \dot{\phi}^\pm}{\p t} - \divg \sigma^\pm(\dot{\phi}^\pm, \dot{\psi}^\pm) = 0 
\quad \text{and} \quad
\divg \dot{\phi}^\pm = 0 & &  \text{in } \Omega^\pm \times (0,T),\\[5pt]
\dot{\phi}^+ = 0 & & \text{on } \p \Omega \times (0,T), \\[5pt]
\dot{\phi}^\pm =  0 
\quad \text{and} \quad
\left[ \sigma (\dot{\phi}, \dot{\psi})\right] n =  0 & & \text{on } \Gamma \times (0,T), \\[5pt]
\phi^\pm(T) = \phi^\pm_T \text{ in }\Omega^\pm, \quad  \zeta(T) = Z_T \text{ on }\Gamma.
\end{array}
\end{equation*}
Let us show that $\dot{\phi}^\pm$ are equal to $0$. The functions $\dot{\phi}^\pm$ can be decomposed in terms of eigenfunctions of the respective Stokes operators~$A^\pm$ in~$\Omega^\pm$, so that the unique continuation argument reduces to show that if
\begin{equation}
\begin{array}{rcl}
- \divg \sigma^\pm(v^\pm, q^\pm) = \lambda^\pm v^\pm 
\quad \text{and} \quad
\divg v^\pm = 0 & &  \text{in } \Omega^\pm\\[5pt]
v^+ = 0 & & \text{on } \p \Omega \\[5pt]
v^\pm =  0 
\quad \text{and} \quad
\left[ \sigma (v, q)\right]n =  0 & & \text{on } \Gamma, \\
\end{array} \label{sys-rad}
\end{equation}
for a given $\lambda^\pm \in \mathrm{Sp}(A^\pm)$, then $v^\pm \equiv 0$. Let us first show that $v^- \equiv 0$ in $\Omega^-$. Since $\Omega^-$ is a disk that we can assume to be centered at~$0$, up to a translation, and we can represent the functions $v^- =(v_r,v_\theta)$ and $q^-$ in polar coordinates, using the polar representation of the Laplace/gradient/divergence operators: 
\begin{equation*}
\begin{array} {rcl}
\divg \sigma (v,q) & = & \displaystyle  \nu\left( \begin{matrix}
\displaystyle
\frac{\p^2 v_r}{\p r^2} + \frac{1}{r^2} \frac{\p^2 v_r}{\p \theta^2} + 
\frac{1}{r} \frac{\p v_r}{\p r} - \frac{2}{r^2}\frac{\p v_{\theta}}{\p \theta} -\frac{v_r}{r^2}
 \\[10pt]
\displaystyle \displaystyle 
\frac{\p^2 v_{\theta}}{\p r^2} + \frac{1}{r^2}\frac{\p^2 v_{\theta}}{\p \theta^2} +
\frac{1}{r}\frac{\p v_{\theta}}{\p r} + \frac{2}{r^2}\frac{\p v_r}{\p \theta} 
- \frac{v_{\theta}}{r^2}
\end{matrix} \right)
- \left(\begin{matrix}
\displaystyle \frac{\p q}{\p r}\\[10pt]
\displaystyle\frac{1}{r}\frac{\p q}{\p \theta}
\end{matrix}\right),
\\[25pt]
\divg v & = & \displaystyle
\frac{1}{r} \frac{\p (r v_r)}{\p r} + \frac{1}{r}\frac{\p v_{\theta}}{\p \theta}. 
\end{array}
\end{equation*} 
Furthermore, in view of the radial symmetry of system~\eqref{sys-rad}, these functions have no angular dependence, and thus satisfy
\begin{subequations}
\begin{eqnarray}
\displaystyle-\nu^- \left(\frac{\p^2 v_r}{\p r^2} + \frac{1}{r} \frac{\p v_r}{\p r} -\frac{v_r}{r^2}\right) + \frac{\p q^-}{\p r} = \lambda^- v_r 
\quad \text{and} \quad 
\frac{1}{r} \frac{\p (r v_r)}{\p r} = 0 & & \text{in } \Omega^- 
\label{sys-rad21},\\
\displaystyle-\nu^- \left(\frac{\p^2 v_\theta}{\p r^2} + \frac{1}{r} \frac{\p v_\theta}{\p r} - \frac{v_\theta}{r^2} \right) = \lambda^- v_\theta 
& & \text{in } \Omega^- 
\label{sys-rad22},\\
(v_r,v_\theta) = 0 & & \text{on } \Gamma. 
\label{sys-rad23}
\label{sys-rad2}
\end{eqnarray}
\end{subequations}
We deduce from the second equation of~\eqref{sys-rad21}, and~\eqref{sys-rad23}, that $v_r \equiv 0$, and therefore the first equation of~\eqref{sys-rad2} shows that $q^-$ is constant. Equation~\eqref{sys-rad22} shows that $v_\theta$ satisfies the following Bessel type equation:
\begin{equation*}
r^2\frac{\p^2 v_\theta}{\p r^2} + r \frac{\p v_\theta}{\p r} + \left( \frac{\lambda^-}{\nu^-}r^2-1 \right) v_\theta =0.
\end{equation*}
Via the change of variable $\tilde{r} = \displaystyle \sqrt{\frac{\lambda^-}{\nu^-}}r$, we deduce that there exists $\alpha > 0$ such that
\begin{equation*}
v_\theta(r,\theta)  =  \alpha J_1\left(\sqrt{\frac{\lambda^-}{\nu^-}}r \right),
\end{equation*}
where $J_1$ is the Bessel function of the first kind given by $\displaystyle
J_1(x)  =  \sum_{m=0}^\infty \frac{(-1)^m}{m! (m+1)!}\left(\frac{x}{2}\right)^{2m+1}$.
From the boundary condition given in~\eqref{sys-rad2}, we have $v_\theta(r_s,\theta) =0$, and consequently, if $r_s$ does not lie in the set introduced in~\eqref{def-excluded}, then necessarily $\alpha = 0$ and thus $v_\theta \equiv 0$, and so $v^- \equiv 0$. From there system~\eqref{sys-rad} reduces to
\begin{equation*}
\begin{array}{rcl}
-\divg \sigma^+(v^+ , q^+) = \lambda^+ v^+ 
\quad \text{and} \quad
\divg v^+ = 0 & & \text{in } \Omega^+, \\[5pt]
v^+ = 0 & & \text{on } \p \Omega^+, \\[5pt]
\sigma^+(v^+,q^+)n = q^-n & & \text{on } \Gamma,
\end{array}
\end{equation*}
where $q^-$ is constant. The continuation argument of~\cite{Fabre} applies in order to deduce that $v^+ \equiv 0$ too. Thus $\displaystyle \dot{\phi}^\pm=\frac{\p \phi^\pm}{\p t} \equiv 0$, that is~$\phi^\pm$ are constant in time, and then system~\eqref{sysadj} reduces to 
\begin{equation*}
\begin{array}{rcl}
-\divg \sigma^\pm(\phi_T^\pm, \psi_T^\pm) = 0 
\quad \text{and} \quad
\divg \phi_T^\pm = 0 & & \text{in } \Omega^\pm,\\[5pt]
\phi_T^\pm = 0
\quad \text{and} \quad
-\left[\sigma(\phi_T,\psi_T)\right]n = \mu \Delta_{\Gamma} Z_T & & \text{on } \Gamma.
\end{array}
\end{equation*}
An energy estimate shows that
\begin{equation}
\nu^+ \int_{\Omega^+} |\varepsilon(\phi_T^+) |^2\d \Omega^+
+ \nu^- \int_{\Omega^+} |\varepsilon(\phi_T^-) |^2\d \Omega^-
 + \mu \int_{\Gamma} |\nabla_{\Gamma} Z_T|^2\d \Gamma = 0,
\label{last-energy}
\end{equation}
and yields $\varepsilon(\Phi_T^\pm) \equiv 0$. From~\cite[page~18]{Temam}, there exists $h^\pm \in \R^2$ and $\omega^\pm\in \R$ such that the functions $\phi^\pm_T$ are reduced to $
\phi^\pm_T(y) = h^\pm + \omega^\pm y^{\perp}$, for $y \in \Omega^\pm$,
and since $\phi_T^\pm = 0$ on $\Gamma$, we have $\phi_T^\pm \equiv 0$. Finally, from~\eqref{last-energy} we deduce that $\nabla_{\Gamma} Z_T = 0$, meaning that $Z_T \in \R^2$ is a translation, in particular $Z_T=0$ in $\HH^{1/2}(\Gamma)/\R^2$, which completes the proof.
\end{proof}

\printbibliography

\end{document}